\documentclass[a4paper,reqno]{amsart}
\usepackage[T1]{fontenc}
\usepackage{lmodern}
\usepackage{comment}
\usepackage{amsmath,amssymb,amsthm,mathtools}
\usepackage{booktabs,array}
\usepackage{microtype}
\usepackage{xcolor}
\usepackage[colorlinks=true,linkcolor=blue!50!black,urlcolor=blue!50!black,citecolor=blue!50!black]{hyperref}
\usepackage{dashrule}
\usepackage{mathtools}
\hypersetup{pdftitle={Infinite families of cohomologically nonrigid toric fourfolds}}
\newtheorem{theorem}{Theorem}[section]
\newtheorem*{theorem*}{Theorem}
\newtheorem{proposition}[theorem]{Proposition}
\newtheorem{lemma}[theorem]{Lemma}
\newtheorem{corollary}[theorem]{Corollary}
\theoremstyle{definition}
\newtheorem{remark}[theorem]{Remark}

\newtheorem{definition}[theorem]{Definition}
\newtheorem{notation}[theorem]{Notation}

\newtheorem*{question*}{Question}

\let\oldtocsection=\tocsection 
\let\oldtocsubsection=\tocsubsection 
\renewcommand{\tocsection}[2]{\hspace{0mm}\oldtocsection{#1}{#2}}
\renewcommand{\tocsubsection}[2]{\hspace{6mm}\oldtocsubsection{#1}{#2}}

\usepackage{tikz-cd}
\tikzcdset{arrow style=math font}
\usetikzlibrary{external,arrows.meta}

\newcommand{\Z}{\mathbb Z}
\newcommand{\Q}{\mathbb Q}
\newcommand{\R}{\mathbb R}
\newcommand{\C}{\mathbb C}

\newcommand{\one}{\mathbf 1}
\newcommand{\ZZ}{\operatorname{Z}}
\newcommand{\SZ}{\operatorname{SZ}}

\newcommand{\cone}{\mathrm{Cone}}

\newcommand{\TG}[1]{{\color{purple}#1}}
\newcommand{\xin}[1]{{\color{teal}#1}}

\definecolor{tcan}{RGB}{0,120,180}
\newcommand{\yingxin}[1]{{\color{tcan}#1}}

\numberwithin{equation}{section}
\title{Cohomological non-rigidity of toric manifolds}

\author{Xin Fu}
\address{
Shanghai Institute for Mathematics and Interdisciplinary Sciences (SIMIS)}
\email{x.fu@simis.cn}
\thanks{Xin Fu is supported by the National Natural Science Foundation of
China (grant no.\,12501083) and the startup fund of SIMIS}

\author{Tao Gong}
\address{
Department of Mathematics, University of Western Ontario,
1151 Richmond Street, London, Ontario, N6A~5B7, Canada}
\email{tgong23@uwo.ca}

\author{Sergei O. Ivanov}
\address{Beijing Key Laboratory of Topological Statistics and Applications
for Complex Systems,
Beijing Institute of Mathematical Sciences and Applications (BIMSA),
Beijing 101408, China.}
\email{ivanov.s.o.1986@gmail.com, ivanov.s.o.1986@bimsa.cn}

\author{Yingxin Li}
\address{
School of Mathematical Sciences, Nankai University, Tianjin, China}
\email{yingxinli@nankai.edu.cn}



\date{\today}

\begin{document}

\begin{abstract}
For every integer $d\geq4$, we construct pairs of projective toric manifolds of complex dimension $d$ that are not homotopy equivalent but have isomorphic integral cohomology rings and homotopy equivalent 5-fold suspensions. We also construct pairs of quasitoric manifolds over the $d$-dimensional cube with these properties. Our construction starts with a closed, oriented smooth manifold $B$ of real dimension $6$ whose integral cohomology ring is generated by $H^2(B)$. For a class $u\in H^2(B)$, let $L_u$ denote the complex line bundle with first Chern class $u$, and set $P_u=\mathbb P_B(L_0\oplus L_u)$. We show that if~$u,v\in H^2(B)$ satisfy $uv=0$, then, for every finite product $N$ of complex projective spaces, the manifolds $P_{u+v}\times N$ and $P_{u-v}\times N$ have isomorphic integral cohomology rings and homotopy equivalent 5-fold suspensions. 
To distinguish their homotopy types, we introduce the sets
of strong zero-product pairs
$\mathrm{SZ}^{n,m}(X)\subseteq H^n(X)\times H^m(X)$,
which are preserved by homotopy equivalences. 
Using these sets, we give explicit conditions on $u$ and $v$ ensuring that these manifolds are not homotopy equivalent, and exhibit toric and quasitoric examples satisfying
these conditions.
\end{abstract}

\maketitle

\tableofcontents

\section{Introduction}

The cohomological rigidity problem asks whether the
integral cohomology ring determines a manifold up to homeomorphism or diffeomorphism within a given
class. In toric topology, this question has been studied for toric manifolds and quasitoric manifolds.
Here toric manifolds are smooth
complete complex toric varieties, and quasitoric manifolds are their topological analogue introduced in~\cite{DJ91}.
Both classes include projective toric manifolds.
The cohomology rings of these manifolds admit explicit
descriptions in terms of their fans or characteristic
data; see~\cite{Dan78,Jur80,DJ91} for more details.

Early formulations of the cohomological rigidity problem in toric topology appear in~\cite[Questions~2.3 and~5.8]{MP08} and~\cite[Problem~1]{MS08}.
 A stronger question asks
whether every isomorphism between the integral cohomology rings of two manifolds in the given class is induced by a homeomorphism or a diffeomorphism; a homotopy version asks whether an isomorphism of
integral cohomology rings implies homotopy equivalence.
For early surveys of these questions and their development,
see~\cite{MS08,CMS11}, and for general background on toric
topology, see~\cite{BP15}.

Affirmative answers are known for several families, including
Bott manifolds~\cite{CHJ25},
two-stage generalized Bott manifolds~\cite{CMS10}, smooth toric Fano manifolds of complex dimension at most $4$~\cite{HKM22,Cho26},
toric and quasitoric manifolds with $b_2=2$~\cite{CP17}, and quasitoric manifolds over simplices~\cite{BP2000}, the $3$-cube~\cite{Ha}
or Pogorelov polytopes~\cite{BEMPP17}.
Further results concern toric manifolds whose underlying
simplicial complexes are flag and satisfy the separable
circuit condition~\cite{FMW24}.
Related results concern homotopy equivalence after localization~\cite{FSST25, FSST26}, and $p$-local stable homotopy equivalence under suitable dimension
bounds~\cite{HK}.

In this paper, we construct explicit counterexamples to the cohomological rigidity conjecture for projective toric manifolds of complex dimension $d$ and for quasitoric manifolds over the $d$-dimensional cube, for every $d\geq 4$. The corresponding rigidity problem for projective toric manifolds of complex dimension $3$ remains open.

For a CW-complex $X$, we write $H^*(X)=H^*(X;\Z)$. For any $u\in H^2(X)$, let $L_u$ denote the complex line bundle over $X$ with first Chern class $u$, and write
$\one=L_0$ for the trivial complex line bundle. Consider the projective bundles
\[
P_u=\mathbb P(\one\oplus L_u).
\]
If $u,v\in H^2(X)$ satisfy $uv=0$, then the cohomology ring formula of projective bundles shows that $H^*(P_{u+v})\cong H^*(P_{u-v})$; see Proposition~\ref{prop:projective_bundle_cohomology} and Proposition~\ref{prop:projectivization_twist}. The remaining task is to distinguish the homotopy types of these spaces despite this ring isomorphism.

For this purpose, we introduce the notion of \emph{strong zero-product pairs}. 
Let 
\[
\ZZ^{n,m}(X)
=
\{(u,v)\in H^n(X)\times H^m(X)\mid uv=0\}
\]
be the set of \emph{zero-product pairs}.
Define $\SZ^{n,m}(X)\subseteq \ZZ^{n,m}(X)$ to be the image 
\[
\SZ^{n,m}(X)=\operatorname{Im}\big([X,K(\Z,n)\vee K(\Z,m)]\to [X, K(\Z,n)\times K(\Z,m)] \big),
\]
of the map induced by the inclusion of the wedge into the
product. Here brackets $[-,-]$ denote the unbased homotopy classes, and $H^n(X)\times H^m(X)$ is identified with $[X, K(\Z,n)\times K(\Z,m)]$. Elements of $\SZ^{n,m}(X)$ are called
\emph{strong zero-product pairs}. 
Every
homotopy equivalence induces an isomorphism on cohomology preserving strong zero-product pairs. Therefore these pairs can be  used to distinguish spaces with isomorphic cohomology rings that are not homotopy equivalent.

In particular, {the pair $(u,v)$ belongs to $\SZ^{2,2}(X) $ if and only if} 
\[
L_u\oplus L_v\cong \one \oplus L_{u+v},
\] 
see Proposition~\ref{prop:SZ:Lu+Lv}.  
This characterization gives a numerical obstruction to a
zero-product pair being strong (Proposition~\ref{prop:obstruction_SZ22}).
Together with an analysis of cohomology ring isomorphisms
between projective bundles, it yields the following criterion.

\begin{theorem*}[Theorem~\ref{th:main'}, Corollary~\ref{cor:suspensions}]
Let $B$ be a closed, connected, oriented smooth manifold
of real dimension $6$ whose integral cohomology ring
is generated by~$H^2(B)$.
Let $u,v\in H^2(B)$ satisfy the following conditions:
\begin{enumerate}
\item $(u,v)\in\ZZ^{2,2}(B)\setminus\SZ^{2,2}(B)$;
\item $(u+v)^2\ne0$;
\item there exists $h\in H^2(B)$ such that
      $(u+v)^2h=0$ and $h^3\ne0$;
\item every graded ring automorphism $\varphi$ of $H^*(B)$
      that fixes $(u-v)^2$ sends $u-v$ to $\pm(u-v)$.
\end{enumerate}
Then, for every finite product of complex projective spaces
$N=\mathbb{CP}^{d_1}\times\cdots\times\mathbb{CP}^{d_k}$,
including the empty product, the smooth manifolds
\[
P_{u-v}\times N,
\qquad\text{and}\qquad
P_{u+v}\times N
\]
are not homotopy equivalent but have isomorphic integral
cohomology rings and homotopy equivalent 5-fold suspensions.
\end{theorem*} 

The cohomology ring isomorphism $H^*(P_{u-v}\times N)\cong H^*(P_{u+v}\times N)$ follows from $uv=0$. To distinguish their homotopy types, we prove that no isomorphisms between their integral cohomology rings  preserve strong zero-product pairs. We then construct two
$6$-dimensional base manifolds satisfying the theorem:
a projective toric manifold with $b_2=4$ and a quasitoric
manifold over the $3$-cube with $b_2=3$.
Taking products of the associated projective bundles
with $\mathbb{CP}^{d-4}$ and $(\mathbb{CP}^1)^{d-4}$,
respectively, gives projective toric counterexamples
of complex dimension $d$ and quasitoric counterexamples
over the $d$-cube for every $d\geq4$. The latter does not contradict the rigidity
of Bott manifolds, which form a proper subclass of quasitoric manifolds over
cubes. Already $\mathbb{CP}^2\#\mathbb{CP}^2$ is a quasitoric manifold over the square that is not homeomorphic
to a Bott manifold; see~\cite[Example~1.19]{DJ91}.

The paper is organized as follows.
In Section~\ref{section SZ}, we introduce strong zero-product pairs,  establish their basic properties and  give a numerical obstruction to a
zero-product pair being strong.
In Section~\ref{section projective bundles}, we review the necessary facts about projective bundles, and determine
whether certain pairs of cohomology classes are strong.
We prove that our manifolds are not homotopy equivalent
in Section~\ref{section Cohomological non-rigidity}
and that their five-fold suspensions are homotopy equivalent
in Section~\ref{section:suspensions}. In Section~\ref{sect:toric and quasitoric}, we review toric manifolds, quasitoric manifolds and their cohomolgy rings.
Finally, in Section~\ref{section:examples}, we construct
the toric and quasitoric examples.


\subsection*{Acknowledgements}

We are grateful to Stephen Theriault, Suyoung Choi,
Matthias Franz, and Mikiya Masuda for helpful discussions.

\subsection*{Use of artificial intelligence} 
The initial version of the example presented in this article was suggested by an AI assistant. We also used AI assistance to refine the construction, check mathematical arguments, and draft and revise the exposition. All arguments and references were independently checked by the authors, who take full responsibility for the mathematical content and references.

\section{Strong zero-product pairs}\label{section SZ}

In this section, for a space $X$, we introduce the sets $\SZ^{n,m}(X)$ of strong zero-product pairs 
 and establish
their elementary properties. We show that
$\SZ^{2,2}(X)$ consists precisely of the pairs
$(u,v)\in H^2(X)\times H^2(X)$ for which
$L_u\oplus L_v$ admits a nowhere-zero continuous section.
Finally, we give a numerical criterion for strong zero-product pairs on closed oriented manifolds of real dimension~6. 

Throughout this section, $X$ denotes a space having the homotopy
type of a CW complex, $n$ and $m$ are nonnegative integers,
and $R$ is a commutative ring. We write $[X,Y]$ for the set of
unbased homotopy classes of maps from $X$ to $Y$.
In particular, we have the natural identification
\[
H^n(X;R)\cong[X,K(R,n)].
\]
The Eilenberg--Mac Lane spaces are equipped with their
standard basepoints, which we use to form wedges
and smash products.

\begin{definition}
A pair of classes
$u\in H^n(X;R)$ and $v\in H^m(X;R)$
is called a \emph{zero-product pair} if $uv=0$.
When we wish to specify the degrees and the coefficient ring,
we call it a zero-product $(n,m)$-pair with coefficients in $R$.
The set of such pairs is denoted by
\[
\ZZ^{n,m}(X;R)
\subseteq H^n(X;R)\times H^m(X;R).
\]
\end{definition}

\begin{definition}\label{definition SZ}
A pair of classes
$u\in H^n(X;R)$ and $v\in H^m(X;R)$
is called a \emph{strong zero-product pair} if it lies
in the image of the natural map
\[
[X,K(R,n)\vee K(R,m)]
\longrightarrow H^n(X;R)\times H^m(X;R)
\]
induced by the inclusion
$K(R,n)\vee K(R,m)\to K(R,n)\times K(R,m)$
and the identification
\[
[X,K(R,n)\times K(R,m)]
\cong H^n(X;R)\times H^m(X;R).
\]
When we wish to specify the degrees and the coefficient ring,
we call it a strong zero-product $(n,m)$-pair
with coefficients in $R$.
The set of such pairs is denoted by
\[
\SZ^{n,m}(X;R)\subseteq H^n(X;R)\times H^m(X;R).
\]
\end{definition}

\begin{remark}[Homotopy invariant]\label{rmk:SZ as a homotopy invariant}
The assignments $\ZZ^{n,m}(-;R)$ and $\SZ^{n,m}(-;R)$
are subfunctors of the contravariant functor
$H^n(-;R)\times H^m(-;R)$ from the homotopy category
of CW complexes to the category of sets.
In particular, $\SZ^{n,m}(X;R)$ is a homotopy invariant of $X$.
\end{remark}

\begin{proposition}[Elementary properties of $\SZ^{n,m}$]\label{prop:elementary properties of SZ}
The following properties hold:
\begin{enumerate}
\item $\SZ^{n,m}(X;R)\subseteq \ZZ^{n,m}(X;R)$.
\item $(u,v)\in\SZ^{n,m}(X;R)$ if and only if
      $(v,u)\in\SZ^{m,n}(X;R)$.
\item If $(u,v)\in\SZ^{n,m}(X;R)$ and $r,s\in R$,
      then $(ru,sv)\in\SZ^{n,m}(X;R)$.
\item $(u,0)\in\SZ^{n,m}(X;R)$
      for every $u\in H^n(X;R)$.
\end{enumerate}
\end{proposition}

\begin{proof}
For (1), the cup product is represented by a map
\[K(R,n)\times K(R,m)
\to K(R,n)\wedge K(R,m)
\xrightarrow{\ \mu\ } K(R,n+m).\]
The composite
$K(R,n)\vee K(R,m)
\to K(R,n)\times K(R,m)
\to K(R,n)\wedge K(R,m)$
is constant. Hence every pair represented by a map
into the wedge has zero cup product.

Properties (2) and (3) follow, respectively,
from the commutative diagrams
\[
\begin{tikzcd}
K(R,n)\vee K(R,m)
\ar[r]
\ar[d,"\mathrm{twist}","\cong"']
&
K(R,n)\times K(R,m)
\ar[d,"\mathrm{twist}","\cong"']
\\
K(R,m)\vee K(R,n)
\ar[r]
&
K(R,m)\times K(R,n)
\end{tikzcd}
\]
and
\[
\begin{tikzcd}
K(R,n)\vee K(R,m)
\ar[r]
\ar[d,"(\cdot r)\vee(\cdot s)"]
&
K(R,n)\times K(R,m)
\ar[d,"(\cdot r)\times(\cdot s)"]
\\
K(R,n)\vee K(R,m)
\ar[r]
&
K(R,n)\times K(R,m)
\end{tikzcd}
\]
where $\cdot r$ and $\cdot s$ denote the maps induced
by multiplication by $r$ and $s$, respectively.

For (4), choose a map $f:X\to K(R,n)$ representing $u$.
Then the composition of $f$ with
$K(R,n)\to K(R,n)\vee K(R,m)\to K(R,n)\times K(R,m)$
represents the pair $(u,0)$.
\end{proof}

\begin{notation}
For integral coefficients, we omit the coefficient ring and write
\[
\SZ^{n,m}(X)=\SZ^{n,m}(X;\Z),\qquad
\ZZ^{n,m}(X)=\ZZ^{n,m}(X;\Z). 
\]
\end{notation}

\begin{notation}
For a CW complex $X$ and a class $u\in H^2(X)$,
we denote by $L_u$ the complex line bundle with the first Chern class
$c_1(L_u)=u$, determined up to isomorphism.
We write $\one=L_0$ for the trivial complex line bundle.
\end{notation}

\begin{proposition}
\label{prop:SZ:Lu+Lv}
Let $X$ be a CW complex and let $u,v\in H^2(X)$.
The following conditions are equivalent:
\begin{enumerate}
\item $(u,v)\in\SZ^{2,2}(X)$;
\item $L_u\oplus L_v$ admits a nowhere-zero continuous section;
\item $L_u\oplus L_v\cong\one\oplus L_{u+v}$
      as complex vector bundles.
\end{enumerate}
\end{proposition}

\begin{proof}
Let $\gamma$ be the universal complex line bundle over
$\mathbb{CP}^{\infty}=K(\Z,2)$, and consider the bundle
\[
E=\operatorname{pr}_1^*\gamma\oplus\operatorname{pr}_2^*\gamma
\]
over $\mathbb{CP}^{\infty}\times\mathbb{CP}^{\infty}$.
Denote by
$p\colon S(E)\to\mathbb{CP}^{\infty}\times\mathbb{CP}^{\infty}$
the associated sphere bundle.
By \cite[Theorem~4.3.2 and Example~4.3.4]{BP15},
there is a homotopy equivalence
\[
\mathbb{CP}^{\infty}\vee\mathbb{CP}^{\infty}
\xrightarrow{\ \simeq\ }S(E)
\]
whose composite with $p$ is homotopic to the canonical inclusion.
Consequently, $\SZ^{2,2}(X)$ is the image of the map
\[
p_*\colon [X,S(E)]\longrightarrow
[X,\mathbb{CP}^{\infty}\times\mathbb{CP}^{\infty}]
\cong H^2(X)\times H^2(X).
\]

Let
$f\colon X\to\mathbb{CP}^{\infty}\times\mathbb{CP}^{\infty}$
classify the pair $(L_u,L_v)$.
Condition~(1) holds if and only if $f$ admits a lift
through $p$ up to homotopy.
Since $p$ is a fibration, this is equivalent to the existence
of an actual lift. Such a lift is equivalent to a section
of the pullback sphere bundle
\[
f^*S(E)\cong S(L_u\oplus L_v),
\]
which, in turn, is equivalent to a nowhere-zero section
of $L_u\oplus L_v$.
This proves $(1)\Longleftrightarrow(2)$.

If (2) holds, the section spans a trivial complex line
subbundle. Choosing a Hermitian metric and taking its
orthogonal complement gives
\[
L_u\oplus L_v\cong\one\oplus L,
\]
where $L$ is a complex line bundle with
\[
c_1(L)=c_1(\one\oplus L)=c_1(L_u\oplus L_v)=u+v.
\]
Hence $L\cong L_{u+v}$, proving (3).
Conversely, the trivial summand in (3) provides a
nowhere-zero section, proving (2).
\end{proof}

\begin{notation}
For a space $X$, let
\[
H^*(X)\longrightarrow H^*(X;\mathbb Z/2), \qquad x\mapsto \overline{x}
\]
denote reduction modulo $2$ in cohomology.
\end{notation}

\begin{notation}
If $B$ is a closed and oriented manifold of real dimension $n$, we write
\[
\int_B\colon H^n(B)\longrightarrow\Z,
\qquad
\alpha\longmapsto\langle\alpha,[B]\rangle,
\]
for evaluation on the fundamental class $[B]\in H_n(B)$.
\end{notation}

\begin{proposition}
\label{prop:obstruction_SZ22}
Let $B$ be a closed oriented smooth manifold of real
dimension~$6$, and let $(u,v)\in\ZZ^{2,2}(B)$.
Suppose
\[
\overline{u+v}=w_2(B)
\qquad\text{and}\qquad
\int_B\bigl(up_1(B)-u^3\bigr)\not\equiv0\pmod{48},
\]
where $w_2(B)$ denotes the second Stiefel-Whitney class, and $p_1(B)$ denotes the first Pontryagin class. Then $(u,v)\notin\SZ^{2,2}(B)$.
\end{proposition}

\begin{proof}
Suppose that $(u,v)\in\SZ^{2,2}(B)$.
By Proposition~\ref{prop:SZ:Lu+Lv}, the bundle
$L_u\oplus L_v$ admits a nowhere-zero continuous section
$(s_1,s_2)$.
By compactness of $B$, we may approximate this section
by a smooth nowhere-zero section and make a sufficiently
small perturbation of $s_1$ so that it is transverse to
the zero section while $(s_1,s_2)$ remains nowhere zero.
The zero set
\[
D=s_1^{-1}(0),\qquad i\colon D\hookrightarrow B,
\]
is a closed smooth manifold of real dimension $4$,
possibly disconnected. 
By transversality, the normal bundle of $D$ is naturally
isomorphic to $(i^*L_u)_{\mathbb R}$.
Thus
\[
i^*TB\cong TD\oplus(i^*L_u)_{\mathbb R}.
\]
The restriction $s_2|_D$ is nowhere zero and therefore
trivializes $i^*L_v$. Hence $i^*v=0$.


Since $\overline{u+v}=w_2(B)$, the Whitney sum formula gives the first two  Stiefel--Whitney classes 
\begin{align*}
    w_1(D)&=i^*\bigl(w_1(B)-w_1(L_u)\bigr)=0,\\
    w_2(D)&=i^*\bigl(w_2(B)-\overline{u}\bigr)=i^*\overline{v}=0.
\end{align*}
Thus every component of $D$ is orientable. Give $D$ the orientation determined by the orientation
of $B$ and the complex orientation of its normal bundle. Each component of $D$ then admits a spin structure.

With this orientation, the submanifold $D$ represents the Poincar\'e dual of the Euler class $e((L_u)_{\R})=u$; see~\cite[\S9 and Problem~11-C]{MS74}. Hence $i_*[D]=u\frown[B]$, where $\frown$ denotes the cap product.
Furthermore, 
Rokhlin's theorem~\cite{Rok52} implies that the signature 
$\operatorname{sign}(D)$ of $D$ is divisible by $16$.

Since the first Pontryagin class $p_1(L_u)=c_1(L_u)^2-2c_2(L_u)=u^2$, the tangent-bundle
splitting and the Hirzebruch signature theorem
(see for example \cite[Theorem~13.7]{Roe98}) give
\[
\int_B\bigl(up_1(B)-u^3\bigr)
=\int_D i^*\bigl(p_1(B)-u^2\bigr)=\int_D p_1(D)
=3\operatorname{sign}(D).
\]
Consequently, this integral is divisible by $48$,
contrary to the hypothesis.
\end{proof}

\section{Projective bundles}\label{section projective bundles}

In this section, we recall the construction of projective bundles associated with vector bundles and describe their integral cohomology rings. For a smooth manifold~$B$ and $\alpha\in H^2(B)$, we introduce the bundle $P_\alpha$ over $B$ and study the behaviour of strong zero-product pairs on $P_{\alpha}$.

\begin{notation}[Projectivisation of a vector bundle]
Let $B$ be a smooth manifold and let $p\colon E\to B$
be a smooth complex vector bundle of rank $r\geq1$.
The \emph{projectivisation} of $E$ is the projective bundle
\[
\pi\colon\mathbb P_B(E)\longrightarrow B,
\]
whose fibre over $b\in B$ is the projective space of complex
$1$-dimensional subspaces of $E_b$.
The \emph{tautological line bundle} and its dual are denoted by
\[
\mathcal O(-1)\longrightarrow\mathbb P_B(E),
\qquad
\mathcal O(1)\longrightarrow\mathbb P_B(E),
\]
respectively. Here $\mathcal O(-1)$ is the line subbundle
of $\pi^*E$ whose fibre over a point $(b,\ell)$ is the line
$\ell$, and $\mathcal O(1)=\mathcal O(-1)^*$.
\end{notation}

\begin{proposition}[See {\cite[Ch.~24, \S3]{May99}}]
\label{prop:projective_bundle_cohomology}
Let $B$ be a smooth manifold and let $E\to B$ be a smooth
complex vector bundle of rank $r\geq1$.
Then the homomorphism
\[
\pi^*\colon H^*(B)\longrightarrow H^*(\mathbb P_B(E))
\]
is injective. Moreover, there is an isomorphism of graded
$H^*(B)$-algebras
\begin{equation}
\label{eq:cohomology_of_projectivisations}
H^*(\mathbb P_B(E))
\cong
H^*(B)[z]\big/
\big\langle z^r+c_1(E)z^{r-1}+\cdots+c_r(E)\big\rangle,
\end{equation}
where the polynomial variable $z$ has degree two
and corresponds to $c_1(\mathcal{O}(1))$.
\end{proposition}

Henceforth, we identify $H^*(B)$ with its image in $H^*(\mathbb P_B(E))$.

\begin{proposition}[See {\cite[Lemma~2.1]{CMS10}}]
\label{prop:projectivization_twist}
Let $B$ be a smooth manifold, $E\to B$  a smooth
complex vector bundle, and  $L\to B$
 a smooth complex line bundle.
Then there is a natural isomorphism of smooth fibre bundles
over $B$
\[
\mathbb P_B(E)\cong\mathbb P_B(E\otimes L).
\]
\end{proposition}

\begin{notation}\label{notation projective bundle and section}
For a manifold $B$ and a class $\alpha\in H^2(B)$, we write
\[
P_\alpha=\mathbb P_B(\one\oplus L_\alpha)
\]
and denote its bundle projection by
$\pi_\alpha\colon P_\alpha\to B$.
This projection has two canonical sections
\[
\sigma\colon B\longrightarrow P_\alpha,
\qquad
\tau\colon B\longrightarrow P_\alpha,
\]
which send $b\in B$ to the lines
$\one_b\oplus\{0\}$ and $\{0\}\oplus(L_\alpha)_b$,
respectively.
We denote the first Chern class of the dual tautological
line bundle by
\[
z_{\alpha}=c_1\bigl(\mathcal O(1)\bigr)\in H^2(P_{\alpha}).
\]
\end{notation}

\begin{corollary}\label{Corollary cohomology ring of P_alpha}
Let $B$ be a smooth manifold and let $\alpha\in H^2(B)$.
Then
\[
H^*(P_{\alpha})
\cong
H^*(B)[z_\alpha]\big/
\big\langle z_\alpha^2+\alpha z_\alpha\big\rangle,
\]
Moreover, the above section $\sigma$ determined by the trivial summand
satisfies $\sigma^*(z_\alpha)=0$.
\end{corollary}
\begin{proof}
The cohomology ring presentation follows from
Proposition~\ref{prop:projective_bundle_cohomology}.
Since $\sigma^*\mathcal O(1)$ is trivial, we have
$\sigma^*(z_\alpha)=0$.
\end{proof}

\begin{proposition}
\label{prop:tautological_pairs}
Let $B$ be a smooth manifold and let $\alpha\in H^2(B)$.
Then 
\[
(z_\alpha,z_\alpha+\alpha)\in\SZ^{2,2}(P_\alpha).\]
\end{proposition}
\begin{proof}
Recall that any morphism of vector bundles
$f\colon E\to E'$ over $X$ defines a section of
$\operatorname{Hom}(E,E')\cong E^*\otimes E'$,
given by $x\mapsto f_x$, where $f_x\colon E_x\to E'_x$
is the induced linear map on the fibres.
The tautological inclusion
\[
\mathcal O(-1)\hookrightarrow
\pi_\alpha^*(\one\oplus L_\alpha)
\]
therefore defines a nowhere-zero section of
\[
\mathcal O(1)\otimes\pi_\alpha^*(\one\oplus L_\alpha)
\cong
\mathcal O(1)\oplus
\bigl(\mathcal O(1)\otimes\pi_\alpha^*L_\alpha\bigr).
\]
The two summands have first Chern classes
\[
c_1(\mathcal O(1))=z_\alpha,
\qquad
c_1\bigl(\mathcal O(1)\otimes\pi_\alpha^*L_\alpha\bigr)
=z_\alpha+\alpha.
\]
The assertion follows from
Proposition~\ref{prop:SZ:Lu+Lv}.
\end{proof}

\section{Cohomological non-rigidity}
\label{section Cohomological non-rigidity}

In this section, we give a general method for constructing
pairs of smooth manifolds with isomorphic integral cohomology
rings that are not homotopy equivalent. When applied to suitable
projective toric bases, this construction produces projective
toric manifolds of every complex dimension $d\geq4$.
Our main result is the following theorem.

\begin{theorem}
\label{th:main'}
Let $B$ be a closed, connected, oriented smooth manifold
of real dimension $6$ whose integral cohomology ring
is generated by $H^2(B)$.
Let $u,v\in H^2(B)$ satisfy the following conditions:
\begin{enumerate}
\item $(u,v)\in\ZZ^{2,2}(B)\setminus\SZ^{2,2}(B)$;
\item $(u+v)^2\ne0$;
\item there exists $h\in H^2(B)$ such that
      $(u+v)^2h=0$ and $h^3\ne0$;
\item every graded ring automorphism $\varphi$ of $H^*(B)$ that fixes $(u-v)^2$ sends $u-v$ to $\pm(u-v)$.
\end{enumerate}
Then, for every finite product of complex projective spaces
$N=\mathbb{CP}^{d_1}\times\cdots\times\mathbb{CP}^{d_k}
$, including the empty product, the smooth manifolds
\[
P_{u-v}\times N
\qquad\text{and}\qquad
P_{u+v}\times N
\]
have isomorphic integral cohomology rings but are not
homotopy equivalent.
\end{theorem}

The remainder of this section is devoted to the proof
of Theorem~\ref{th:main'}, which we will separate into several lemmas. Throughout, $B$ and $N$ are as in Theorem~\ref{th:main'}.

Since $H^*(B)$ is generated by $H^2(B)$, its odd degree
cohomology groups vanish. As $B$ is a closed, oriented
manifold of even dimension, Poincar\'e duality implies
that its odd degree homology groups also vanish.
The universal coefficient theorem therefore gives
$H^{2k}(B)\cong\operatorname{Hom}(H_{2k}(B),\mathbb Z)$.
Thus $H^*(B)$ is torsion-free. 

For $\alpha\in H^2(B)$, 
it follows from Corollary~\ref{Corollary cohomology ring of P_alpha} that 
\[H^*(P_{\alpha})\cong H^*(B)[z_{\alpha}]/\langle z_{\alpha}^2+\alpha z_{\alpha}\rangle.\] 
We identify $H^*(B)$ with its image under the pullback in  $H^*(P_{\alpha})$.
The K\"unneth formula shows that 
\begin{equation}\label{eq cohomology ring}
\begin{split}
H^*(P_{\alpha}\times N)&\cong H^*(P_\alpha)\otimes H^*(N)\\
&\cong H^*(B)[z_{\alpha},x_i|1\leq i\leq k]\big/\langle z_{\alpha}^2+\alpha z_{\alpha},\ x_i^{d_i+1}\rangle,  
\end{split}
\end{equation}
where $x_i$ is the pullback of the degree $2$ generator of $H^*(\mathbb{CP}^{d_i})\cong \Z[x]\big/\langle x^{d_i+1}\rangle$.

Let $\zeta_\alpha= z_\alpha + \frac{\alpha}{2} \in H^2(P_\alpha;\mathbb Q)$. 
Then $\zeta_\alpha^2=\frac{\alpha^2}{4}\in H^*(B,\mathbb{Q})$, and 
\[
    H^2(P_\alpha;\Q) \cong H^2(B;\Q)\oplus \Q\zeta_\alpha.
\]
For $\beta\in H^2(B;\Q)$ and $t\in \Q$, define {the homogeneous polynomial functions}
\begin{align}\label{eq:def of intersection forms}
\ell_\alpha(\beta)=\int_B \alpha^2\beta,
\qquad 
V_B(\beta)=\int_B \beta^3,
    \qquad 
Q_\alpha(\beta,t)=\int_{P_\alpha} (\beta+\zeta_\alpha t)^4,
\end{align} 
as the pairings of the cohomological classes and the fundamental homological classes.
These polynomials have homogeneous degrees $1$, $3$ and $4$, respectively. 
Note that~$V_B(\beta)$ is the volume polynomial of $B$, and 
$Q_\alpha(\beta,t)=V_{P_\alpha}(\beta+t\zeta_\alpha)$
is the volume polynomial of $P_{\alpha}$.

\begin{lemma}\label{lem:cal of Q_w}
For all $\beta\in H^2(B;\Q)$ and $t\in\Q$, we have
\[
    Q_\alpha(\beta,t) = t\left(4V_B(\beta)+t^2\ell_\alpha(\beta)\right).
\]
\end{lemma}
\begin{proof}
Since terms coming from $H^*(B)$ vanish in degree 8, the expansion shows 
  \[\begin{split}
    (\beta+t\zeta_\alpha)^4 &=   \beta^4+4t\beta^3\zeta_\alpha+6t^2\beta^2\zeta_\alpha^2+4t^3\beta\zeta_\alpha^3+t^4\zeta_\alpha^4\\
   & = 4t\beta^3\zeta_\alpha+4t^3\beta\zeta_\alpha^3 \\
    &=(4t\beta^3+t^3\beta\alpha^2)\zeta_\alpha.
  \end{split}\]

Write $\pi\colon P_\alpha\to B$ for the projection.
Let $\pi_!\colon H^*(P_\alpha)\to H^{*-2}(B)$ be the  Gysin map defined by $\pi_!=D_B^{-1}\circ \pi_* \circ D_{P_\alpha}$, where $D_B$ and $D_{P_\alpha}$ denote the Poincar\'e duality isomorphisms of $B$ and $P_\alpha$, respectively.
Over $\Q$, we have
\[
  \begin{aligned}
    Q_\alpha(\beta,t)&=\int_{P_\alpha} \left(4t\beta^3+t^3\beta\alpha^2\right)\zeta_\alpha \\
    &= \int_B \pi_!(\left(4t\beta^3+t^3\beta\alpha^2\right)\zeta_\alpha)\\
       & = \int_B  \left(4t\beta^3+t^3\beta\alpha^2\right) \pi_!(\zeta_\alpha).
  \end{aligned}
 \]
Here the second identity follows from the definition of $\pi_!$, and the last is due to~\cite[(1.3)]{Nak80}.

To compute $\pi_!(\zeta_\alpha)$, fix a point $x\in B$ and consider
the pullback diagram 
\[
\begin{array}{ccc}
\mathbb C P^1 & \xrightarrow{\ i_x\ }   & P_\alpha\\
{\scriptstyle p_x}\downarrow
&&  \downarrow{\scriptstyle\pi}\\
\{x\}   &   \xrightarrow{\ j_x\ }   &   B.
\end{array}
\]
The naturality of Gysin maps (see for instance \cite[(1.4)]{Nak80}) implies 
    \[ 
        j_x^*\pi_!(\zeta_\alpha)= (p_x)_!i_x^*(\zeta_\alpha).
    \]
    Since \( \zeta_\alpha=z_\alpha+\frac{\alpha}{2} \), the class \( i_x^*(\zeta_\alpha) \) is the first Chern class of the dual of the tautological line bundle of \( \mathbb C P^1\).  
    It follows 
    \(\pi_!(\zeta_\alpha)=(j_x^*)^{-1} \int_{\mathbb C P^1}i_x^*(\zeta_\alpha)=1\in H^0(B)\) and therefore
\[
Q_\alpha(\beta,t)=\int_B\left(4t\beta^3+t^3\beta \alpha^2\right)
=t\left(4V_B(\beta)+t^2\ell_\alpha(\beta)\right).\qedhere
\]
\end{proof}

\begin{lemma}\label{lem:iso preserves base}
Let $\alpha_1,\alpha_2\in H^2(B)$. Suppose that $t$ is the unique rational linear factor of $Q_{\alpha_2}$ up to a nonzero rational scalar.
Then every graded ring isomorphism $H^*(P_{\alpha_1})\to H^*(P_{\alpha_2})$ restricts to an automorphism of $H^*(B)$.
\end{lemma}

\begin{proof}

   Let $\Phi\colon H^*(P_{\alpha_1})\to H^*(P_{\alpha_2})$ be a graded ring isomorphism. Then it induces an isomorphism between the rational cohomology rings 
   \[
    \Phi_\Q\colon H^*(P_{\alpha_1};\Q)\to H^*(P_{\alpha_2};\Q),
   \]
   which sends $Q_{\alpha_1}$ to $Q_{\alpha_2}$ up to a nonzero scalar. 
   The image of $H^2(B;\Q)$ under~$\Phi_\Q$ must be a hyperplane on which $Q_{\alpha_2}$ vanishes as $\beta^4=0$ for all $\beta\in H^2(B;\Q)$. 
Under the assumption that $t$ is the unique rational linear factor of $Q_{\alpha_2}$ up to a nonzero scalar, the hyperplane inside $H^2(P_{\alpha_2};\Q)$ must satisfy $t=0$, which is exactly $H^2(B;\Q)$. In other words, the map $\Phi_\Q$ sends $H^2(B;\Q)$ into $H^2(B;\Q)$. In fact, this is onto,
since~$\Phi_\Q$ is injective and the two spaces have the same dimension.

Finally, following from the fact $H^2(P_{\alpha_i}) \cap H^2(B;\Q) = H^2(B)$, we have 
\[\Phi(H^2(B))=H^2(B)\subset H^2(P_{\alpha_2}).\]
Since $H^*(B)$ is generated by $H^2(B)$, it follows that $\Phi(H^*(B))=H^*(B)$. Hence $\Phi$ restricts to a graded
ring automorphism of $H^*(B)$.
\end{proof}

\begin{remark}\label{remark:main conditions}
Conditions~(2) and (3) of Theorem~\ref{th:main'} ensure that
Lemma~\ref{lem:iso preserves base} holds for
$\alpha_1=u-v$ and $\alpha_2=u+v$.
Indeed, since $uv=0$, we have $(u-v)^2=(u+v)^2$.
For $i=1$ or $2$, Condition~(2) and
Poincar\'e duality imply that $\ell_{\alpha_i}\not\equiv0$. Here $\ell_{\alpha_i}\not\equiv0$ means that $\ell_{\alpha_i}(\beta)\ne0$ for some $\beta\in H^2(B;\mathbb Q)$.
 Condition~(3) gives
\[
\ell_{\alpha_i}(h)=0,\qquad V_B(h)\ne0.
\]
Thus $\ell_{\alpha_i}(\beta)$ does not divide $V_B(\beta)$ as polynomials in~$\beta$, and
$4V_B(\beta)+t^2\ell_{\alpha_i}(\beta)$ is therefore irreducible
over $\mathbb Q$.
It follows that $t$ is the unique rational linear factor
of~$Q_{\alpha_i}$, up to a nonzero scalar.
\end{remark}

\begin{lemma}\label{lem:preserves product factor}
Let $\alpha_1,\alpha_2\in H^2(B)$ with nonzero squares. Suppose that each \(Q_{\alpha_i}\) has~\(t\) as its unique rational linear factor up to a nonzero scalar. Then every graded ring isomorphism
$H^*(P_{\alpha_1}\times N) \to H^*(P_{\alpha_2}\times N)$ restricts to an automorphism of~$H^*(N)$.
\end{lemma}
\begin{proof}
Recall from~\eqref{eq cohomology ring} that $\{x_i\mid 1\leq i\leq k\}$ is a basis of $H^2(N;\Q)$. Writing $\mu=\sum_{i=1}^k m_i x_i\in H^2(N;\Q)$  and $d=\sum_{i=1}^k d_i$, we have the volume polynomial
\[
V_N(\mu)=\int_{N} \mu^{d}= \frac{d!}{d_1!d_2!\cdots d_k!}\prod_{i=1}^k m_i^{d_i},
\]
which is a monomial in variables $m_1,\dots , m_k$, and has only linear irreducible factors over $\mathbb Q$.

For $\alpha=\alpha_1$ or $\alpha_2$, the decomposition
\[
H^2(P_\alpha\times N;\Q) \cong H^2(B;\Q) \oplus \Q \zeta_\alpha \oplus H^2(N;\Q) 
\]
identifies a class $\beta+t\zeta_\alpha+\mu$ with
the triple $(\beta,t,\mu)$, where
$\beta\in H^2(B;\Q)$, $t\in\Q$, and
$\mu\in H^2(N;\Q)$. The volume polynomial of $P_\alpha\times N$ satisfies   
\begin{equation}\label{eq volume polynomial of PM}
\begin{aligned}
\int_{P_\alpha\times N}(\beta+t\zeta_\alpha+ \mu)^{d+4} 
& =\binom{d+4}{4}  Q_\alpha(\beta,t)\, V_N(\mu).
\end{aligned}
\end{equation}
Now Lemma~\ref{lem:cal of Q_w} shows that $Q_\alpha(\beta,t)
=t(4V_B(\beta)+t^2\ell_\alpha(\beta))$.
Since $t$ is the unique rational linear factor of $Q_\alpha$, 
the other factor is irreducible over~$\mathbb Q$. 
Since $V_N$ has only linear irreducible factors, it follows that the polynomial
\[G_\alpha(\beta, t,\mu)
\coloneqq 4V_B(\beta)+t^2\ell_\alpha(\beta)
\]
is the unique irreducible cubic factor of the volume polynomial~\eqref{eq volume polynomial of PM},
up to a nonzero rational scalar.

For
$x=\beta+t\zeta_\alpha+ \mu,~
y=\beta'+t'\zeta_\alpha+\mu'\in H^2(P_\alpha\times N;\mathbb Q)$, let
\begin{align}\label{eq:partial y Gs}
\begin{aligned}
    \partial_yG_\alpha(x)
=\left.\frac{d}{d\lambda}\right|_{\lambda=0}G_\alpha(x+\lambda y)
=12\int_B\beta'\,(\beta)^2
+2t\,t'\,\ell_\alpha(\beta)+t^2\ell_\alpha(\beta').
\end{aligned}
\end{align}
Define
\[
W_\alpha=
\{y\in H^2(P_\alpha\times N;\mathbb Q):\partial_yG_\alpha\equiv0\}.
\]

Let $\Phi\colon H^*(P_{\alpha_1}\times N)\to H^*(P_{\alpha_2}\times N)$ be a graded ring isomorphism, and $\Phi_\Q$ be the induced isomorphism of $\Phi$ on rational cohomology.
Since $\Phi_{\mathbb Q}$ carries
the volume polynomials to one another up to a
nonzero scalar, uniqueness of the cubic factor gives
\[
G_{\alpha_2}(\Phi_\Q(x))=cG_{\alpha_1}(x)\quad\text{for some $c\in\mathbb Q^\times$.}
\]
Differentiating in the direction $y$ yields
\[
\partial_{\Phi_\Q(y)}G_{\alpha_2}(\Phi_\Q(x))=c(\partial_yG_{\alpha_1}(x)).
\]
Since $\Phi_\Q$ is invertible, it follows that
\[
\Phi_{\mathbb Q}(W_{\alpha_1})=W_{\alpha_2}.
\]

We next show $W_\alpha=H^2(N;\mathbb Q)$. Let $y\in W_\alpha$. Setting $t=0$ into~\eqref{eq:partial y Gs} of $\partial_yG_\alpha$ gives 
\[
\int_B\beta'\,\beta^2=0\qquad  \text{for all}\quad \beta\in H^2(B;\Q)
\]
Then for any $\eta,\eta'\in H^2(B;\Q)$, we have 
\[
\int_B \beta'\eta'\eta=\frac{1}{4}\int_B \beta'(\eta+\eta')^2- \frac{1}{4}\int_B \beta'(\eta-\eta')^2=0.
\]
Since $H^*(B)$ is generated by $H^2(B)$, these products $\eta'\eta $ span $H^4(B;\Q)$. Thus Poincar\'e duality   implies $\beta'=0$. 
Hence $\partial_yG_\alpha = 2t'\,t\,\ell_\alpha(\beta)$. Since $\alpha^2\ne0$, Poincar\'e duality implies that
$\ell_\alpha(\beta)=\int_B\alpha^2\beta$ is nonzero
for some $\beta\in H^2(B;\mathbb Q)$.
Thus $\ell_\alpha\not\equiv0$, and
$\partial_yG_\alpha\equiv0$ forces $t'=0$. Conversely, every $y$ in~$ H^2(N;\mathbb Q)$ satisfies
$\partial_yG_\alpha\equiv0$. 
This proves the claim, and hence
\[
\Phi_{\mathbb Q}(H^2(N;\mathbb Q))
=H^2(N;\mathbb Q).
\] 
Since $H^2(P_\alpha\times N)\cap H^2(N;\Q)= H^2(N)$ by  K\"unneth formula,  we conclude that
\[
\Phi(H^2(N))=H^2(N) \quad\text{ and } \quad \Phi(H^*(N))=H^*(N),
\] 
where the second equality follows because $H^*(N)$ is generated by $H^2(N)$.
\end{proof}

Proposition \ref{prop:tautological_pairs} gives $(z_{u-v},z_{u-v}+u-v) \in \SZ^{2,2}(P_{u-v}).$ 
By Remark~\ref{rmk:SZ as a homotopy invariant}, homotopy equivalences preserve strong zero-product pairs. The following results use this property to rule out homotopy equivalences between $P_{u-v}\times N$ and $P_{u+v}\times N$.

\begin{lemma}\label{lemma:case I of non-equiv}
Let $B,u,v$ satisfy the hypotheses of
Theorem~\ref{th:main'}.   Then every graded ring isomorphism $\Phi\colon H^*(P_{u-v})
\to H^*(P_{u+v})$
satisfies
\[\Phi(z_{u-v},z_{u-v}+u-v)\not\in \SZ^{2,2}(P_{u+v}).\]
\end{lemma}
\begin{proof}
We first claim that $\Phi$ restricts to an automorphism~of $H^*(B)$.
Remark~\ref{remark:main conditions} shows that $t$ is the unique rational linear factor of $Q_{u+v}$
up to a nonzero scalar. The claim follows from
Lemma~\ref{lem:iso preserves base}.

Second, suppose $\Phi(z_{u-v})=qz_{u+v}+\beta$ for some $q\in \Z$ and $\beta\in H^2(B)$. Since $\Phi$ is an isomorphism, we have $q\in\{\pm 1\}$. Applying \(\Phi\) to \(z_{u-v}(z_{u-v}+u-v)=0\), and comparing the two components in the decomposition
\[
H^4(P_{u+v})\cong H^4(B)\oplus H^2(B)z_{u+v},
\]
gives $2\beta+\Phi(u-v)=q(u+v)$ and $\beta^2+\beta\Phi(u-v)=0$, which implies
\[\begin{split}
\Phi((u-v)^2)&=\Phi((u-v)^2)+4\beta^2+4\beta\Phi(u-v)=(\Phi(u-v)+2\beta)^2\\
&=(u+v)^2=(u-v)^2.
\end{split}\]
Condition~(4) now gives $\Phi(u-v)=\varepsilon(u-v)$ for some $\varepsilon\in \{\pm 1\}$. Hence 
\[
\beta=\frac{q(u+v)-\varepsilon(u-v)}{2}.
\]

There is a section \(\sigma\colon B\to P_{u+v}\) such that \(\sigma^*(z_{u+v})=0\) (see Corollary~\ref{Corollary cohomology ring of P_alpha}). Then
\[
\begin{aligned}
\sigma^*\Phi(z_{u-v},z_{u-v}+u-v)
&=\left(\frac{q(u+v)-\varepsilon(u-v)}{2},
\frac{q(u+v)+\varepsilon(u-v)}{2}\right).
\end{aligned}
\]
This is \((u,v)\), up to interchanging the entries or multiplying both by \(-1\).
By Proposition~\ref{prop:elementary properties of SZ}\,(2)--(3),
these operations preserve $\SZ^{2,2}(B)$.
Condition~(1) therefore implies
$\sigma^*\Phi(z_{u-v},z_{u-v}+u-v)\notin\SZ^{2,2}(B).$
Since pullbacks preserve strong zero-product pairs, the image
$\Phi(z_{u-v},z_{u-v}+u-v)$ cannot be strong.
\end{proof}

\begin{lemma}\label{lemma:case II of non-equiv}
Let $B,u,v$ satisfy the hypotheses of
Theorem~\ref{th:main'}. Then no graded ring isomorphism 
$\Phi\colon H^*(P_{u-v}\times N)
\to H^*(P_{u+v}\times N)$ 
preserves strong zero-product pairs.
\end{lemma}
\begin{proof}
By Lemma~\ref{lem:preserves product factor}, the isomorphism $\Phi$ 
preserves the ideal generated by~$H^{>0}(N)$ inside $H^*(P_{u+v}\times N)$. Taking the quotient by this ideal, which corresponds to
restriction to~$P_{u+v}\times\{\mathrm{pt}\}$, yields a graded
ring isomorphism $\Phi_0$ fitting into a commutative diagram
\[\begin{tikzcd}
H^*(P_{u-v})\ar[r,"\Phi_0"]\ar[d,"\operatorname{pr}^*"']& H^*(P_{u+v})\\
H^*(P_{u-v}\times N)\ar[r,"\Phi"]&H^*(P_{u+v}\times N)\ar[u,"\iota^*"]
\end{tikzcd}
\]
where $\operatorname{pr}\colon P_{u-v}\times N\to P_{u-v}$
is the projection, and $\iota\colon P_{u+v}\to P_{u+v}\times N$ is the inclusion
at a chosen point of $N$.

By Lemma~\ref{lemma:case I of non-equiv}, there is a strong zero-product pair
$(\alpha,\beta)$ in $H^2(P_{u-v})$ whose image under
$\Phi_0$ is not strong. The pair
$(\operatorname{pr}^*(\alpha),\operatorname{pr}^*(\beta))$
is strong, but its image under $\Phi$ cannot be strong.
Otherwise, pulling back by $\iota$ and using the above
commutative diagram would imply that
$(\Phi_0(\alpha),\Phi_0(\beta))$ is strong.
Thus $\Phi$ does not preserve strong zero-product pairs.
\qedhere
\end{proof}

\begin{proof}[Proof of Theorem~\ref{th:main'}]
Clearly, the assignment $z_{u-v}\mapsto z_{u+v}+v$, together with the identity on $H^*(B)\otimes H^*(N)$, defines an isomorphism 
\[ H^*(P_{u-v}\times N)\longrightarrow H^*(P_{u+v}\times N)\]
of graded algebras over $H^*(B)\otimes H^*(N)$.
However, Lemmas~\ref{lemma:case I of non-equiv} and~\ref{lemma:case II of non-equiv} imply that $P_{u-v}\times N$ and $P_{u+v}\times N$ are not homotopy equivalent.
\end{proof}

\begin{remark}
\label{remark:main:Qtrivial}
Theorem~\ref{th:main'} remains valid when $N$ is a closed, oriented smooth manifold of real dimension $2n$
whose integral cohomology ring is generated by $H^2(N)$
and whose top intersection polynomial
\[
V_N\colon H^2(N;\Q)\longrightarrow\Q,
\quad
\mu\longmapsto\int_N\mu^n,
\]
has no irreducible factor of degree $3$ over $\Q$.
The proof requires minor modifications.

These hypotheses hold when $N$ is any
$\mathbb Q$-trivial generalized Bott manifold;
see~\cite{PS14,CM12}.
By~\cite[Theorem~8.1]{CMS10a}, a quasitoric manifold
over a product of simplices is homeomorphic to a generalized Bott manifold if its rational cohomology ring is isomorphic to that of a product of complex
projective spaces. Theorem~\ref{th:main'} therefore remains valid when
$N$ is any of these quasitoric manifolds.
\end{remark}

\section{Homotopy equivalence after  suspensions}
\label{section:suspensions}

The goal of this section is to prove
the following theorem.

\begin{theorem}
\label{thm:suspensions}
Let $B$ be a closed, connected smooth manifold of real
dimension $n\geq2$ whose integral cohomology is torsion-free,
and let $(u,v)\in\ZZ^{2,2}(B)$.
Equip $P_{u+v}$ and $P_{u-v}$ with arbitrary basepoints.
Then there is a based homotopy equivalence
\[
\Sigma^{n-1}P_{u+v}
\simeq
\Sigma^{n-1}P_{u-v}.
\]
\end{theorem}

Before proving the theorem, we state the following corollary.

\begin{corollary}
\label{cor:suspensions}
Let $B$ be a closed, connected smooth manifold of real
dimension $6$ whose integral cohomology is torsion-free,
and let $(u,v)\in\ZZ^{2,2}(B)$.
Then, for every nonempty CW complex $N$, there is a based
homotopy equivalence
\[
\Sigma^5(P_{u+v}\times N)
\simeq
\Sigma^5(P_{u-v}\times N),
\]
where the products are equipped with arbitrary basepoints.
\end{corollary}

\begin{proof}
First observe that, for any nonempty CW complex $X$, the
based homotopy type of $\Sigma(X,x)$ is independent of
$x\in X$. Indeed, collapsing the suspension arc through $x$
gives a based homotopy equivalence from the unreduced
suspension of $X$, based at a fixed pole, to $\Sigma(X,x)$. Consequently, it suffices to prove the assertion for
basepoints $(p_+,n_0)$ and $(p_-,n_0)$, where
$p_\pm\in P_{u\pm v}$ and $n_0\in N$.
Each $P_s$ admits a finite CW structure, so $P_s\times N$
is a CW complex.
The suspension splitting
\cite[Proposition~4I.1]{Hat02} therefore applies. Suspending it four
more times gives
\begin{equation}
\label{eq:product-suspension-splitting}
\Sigma^5(P_s\times N)
\simeq
\Sigma^5P_s
\vee \Sigma^5N
\vee \bigl((\Sigma^5P_s)\wedge N\bigr),
\end{equation}
for $s=u+v$ and $s=u-v$. By Theorem~\ref{thm:suspensions}, there is a based homotopy
equivalence
$f\colon\Sigma^5P_{u+v}\longrightarrow\Sigma^5P_{u-v}$.
The wedge of $f$, $\operatorname{id}_{\Sigma^5N}$, and
$f\wedge\operatorname{id}_N$ is a based homotopy
equivalence between the right-hand sides of
\eqref{eq:product-suspension-splitting}.
\end{proof}

The remainder of this section is devoted to the proof of Theorem \ref{thm:suspensions}. 

\begin{lemma}
\label{lem:stable_real_bundle_isomorphism}
Let $B$ be a finite CW complex of dimension $n\geq 2$ whose integral cohomology
is torsion-free, and let $(u,v)\in \ZZ^{2,2}(B)$.
Then there exists an isomorphism of real vector bundles over $B$
\[
(L_{u+v})_{\mathbb R}\oplus\varepsilon_{\mathbb R}^{n-1}
\cong
(L_{u-v})_{\mathbb R}\oplus\varepsilon_{\mathbb R}^{n-1},
\]
where $E_{\mathbb R}$ denotes the underlying real
vector bundle of a complex vector bundle $E$, and
$\varepsilon_{\mathbb R}^k=B\times\mathbb R^k$.
\end{lemma}

\begin{proof}
For every $\alpha\in H^2(B)$, the Chern character of $L_\alpha$ is
\[
\operatorname{ch}([L_\alpha])
=e^\alpha=\sum_{j=0}^{\infty}\frac{\alpha^j}{j!}
\in H^{\mathrm{even}}(B;\mathbb Q).
\]
Since $B$ is finite-dimensional, the exponential has
only finitely many nonzero terms.

Consider the class
\[
\alpha=[L_{u+v}]-[L_{u-v}]-[L_v]+[L_{-v}]
\in K^0(B).
\]
Its Chern character satisfies
\[
\operatorname{ch}(\alpha)
=e^{u+v}-e^{u-v}-e^v+e^{-v}=(e^u-1)(e^v-e^{-v})=0.
\]
Indeed, every term in the product is a rational multiple
of $u^iv^j$ with $i,j\geq1$, and all such products vanish
because $(u,v)\in\ZZ^{2,2}(B)$ implies $uv=0$.

Since the integral cohomology of $B$ is torsion-free,
the Chern character is injective by
\cite[\S2.5, second corollary, part~(i)]{AH61}.
Thus $\alpha=0$, and hence
\begin{equation}
\label{eq:line_bundle_K_identity}
[L_{u+v}]-[L_{u-v}]
=
[L_v]-[L_{-v}]
\quad\text{in }K^0(B).
\end{equation}

Forgetting the complex structure induces the
realification homomorphism
\[
K^0(B)\longrightarrow KO^0(B),
\qquad
[E]\longmapsto[E_{\mathbb R}].
\]
The line bundle $L_{-v}$ is isomorphic to the conjugate
bundle $\overline{L_v}$, which has the same underlying
real vector bundle as $L_v$. Applying realification
to \eqref{eq:line_bundle_K_identity}, we therefore obtain
\[
[(L_{u+v})_{\mathbb R}]
=
[(L_{u-v})_{\mathbb R}]
\quad\text{in }KO^0(B).
\]

By the definition of the Grothendieck group, there is
a real vector bundle $G$ over~$B$ such that
\[
(L_{u+v})_{\mathbb R}\oplus G
\cong
(L_{u-v})_{\mathbb R}\oplus G.
\]
Since $B$ is compact Hausdorff, $G$ admits a complementary
bundle $G'$ such that
$G\oplus G'\cong\varepsilon_{\mathbb R}^k$
for some $k\geq0$; see \cite[Proposition~1.4]{Hat17}.
Adding $G'$ to both sides of the preceding isomorphism
shows that $(L_{u+v})_{\mathbb R}$ and
$(L_{u-v})_{\mathbb R}$ are stably isomorphic.

Consequently, the real vector bundles
\[
(L_{u+v})_{\mathbb R}\oplus\varepsilon_{\mathbb R}^{n-1}
\quad\text{and}\quad
(L_{u-v})_{\mathbb R}\oplus\varepsilon_{\mathbb R}^{n-1}
\]
are stably isomorphic and have the same rank
$n+1>\dim B$. The real cancellation theorem
\cite[Chapter~I, Theorem~4.3(ii)]{Wei13}
therefore gives the required isomorphism.
\end{proof}

\begin{lemma}
\label{lem:stable_Thom_equivalence}
Let $B$ be a finite CW complex of dimension $n\geq2$
whose integral cohomology is torsion-free, and let
$(u,v)\in \ZZ^{2,2}(B)$. 
Then there is a pointed homeomorphism
\[
\Sigma^{n-1}\operatorname{Th}(L_{u+v})
\cong
\Sigma^{n-1}\operatorname{Th}(L_{u-v}),
\]
where $\operatorname{Th}(-)$ is the Thom space.
\end{lemma}

\begin{proof}
Since $B$ is compact Hausdorff, the Thom space of a
real vector bundle $E$ over $B$ is the one-point
compactification of its total space.
The total space of
$E\oplus\varepsilon_{\mathbb R}^{n-1}$
is naturally homeomorphic to $E\times\mathbb R^{n-1}$.
Consequently, there are pointed homeomorphisms
\[
\operatorname{Th}(E\oplus\varepsilon_{\mathbb R}^{n-1})
\cong
\operatorname{Th}(E)\wedge S^{n-1}
\cong
\Sigma^{n-1}\operatorname{Th}(E);
\]
see \cite[Chapter~23, \S6, p.~196]{May99}. Here we identify the Thom space of a complex vector
bundle with that of its underlying real vector bundle.
Applying these identifications to the real vector
bundle isomorphism from
Lemma~\ref{lem:stable_real_bundle_isomorphism}
gives the required pointed homeomorphism.
\end{proof}
\begin{remark}
A statement closely related to
Lemma~\ref{lem:projective_bundle_suspension} appears in
\cite[Proposition~7.10]{CR05} under slightly different
assumptions. For the reader's convenience, we state
and prove the version needed here.
\end{remark}

\begin{lemma}[{cf.~\cite[Proposition~7.10]{CR05}}]
\label{lem:projective_bundle_suspension}
Let $B$ be a closed, connected smooth manifold and let
$\alpha\in H^2(B)$. Equip $B$ and $P_\alpha$ with arbitrary basepoints.
Then there is a based homotopy equivalence
\[
\Sigma P_\alpha
\simeq
\Sigma B\vee\Sigma\operatorname{Th}(L_\alpha).
\]
\end{lemma}

\begin{proof}
Let $\tau\colon B\to P_\alpha$ be the section that sends
$b\in B$ to the line
$\{0\}\oplus(L_\alpha)_b\in(P_\alpha)_b$. The section $\tau$ is a closed smooth embedding.
Its tubular neighbourhood shows that $\tau$ is a
cofibration. Let $b_0$ be the chosen basepoint of $B$,
and first equip $P_\alpha$ with the basepoint $\tau(b_0)$.

The complement of $\tau(B)$ is naturally
homeomorphic to the total space of $L_\alpha$:
\[
P_\alpha\setminus \tau(B)\cong L_\alpha.
\]
Under this homeomorphism, a vector
$\xi\in(L_\alpha)_b$ corresponds to the line
$\operatorname{span}_{\mathbb C}\{(1,\xi)\}
\subset\mathbb C\oplus(L_\alpha)_b$.
Since $P_\alpha$ is compact Hausdorff and $\tau(B)$ is closed,
the quotient $P_\alpha/\tau(B)$ is the one-point compactification
of its complement $P_\alpha\setminus \tau(B)$.
Since $B$ is compact Hausdorff, the one-point compactification
of the total space of $L_\alpha$ is homeomorphic to its Thom space;
see \cite[\S4.1, p.~112]{Hat17}.
Thus, under the preceding identification of
$P_\alpha\setminus \tau(B)$ with this total space, we obtain
a pointed homeomorphism
\[
P_\alpha/\tau(B)\cong\operatorname{Th}(L_\alpha);
\]
see also \cite[Chapter~24, \S3, p.~210]{May99}.

Since $\tau$ is a cofibration, we obtain a homotopy
cofibre sequence
\[
B\xrightarrow{\ \tau\ }P_\alpha
\xrightarrow{\ q_\alpha\ }\operatorname{Th}(L_\alpha),
\]
where $q_\alpha$ is the quotient map.
Since $\tau$ admits the based retraction $\pi_\alpha$, the suspension
splitting of \cite[\S4.I, Exercise~1, p.~469]{Hat02} gives
\[
\Sigma P_\alpha
\simeq
\Sigma B\vee\Sigma\operatorname{Th}(L_\alpha).
\]
Finally, since $P_\alpha$ is a connected smooth manifold,
any other basepoint can be carried to $\tau(b_0)$ by
a diffeomorphism. Hence the conclusion holds for
arbitrary choices of basepoints.
\end{proof}

\begin{proof}[Proof of Theorem~\ref{thm:suspensions}]
Since $B$ is a closed smooth manifold of dimension $n$,
it admits a finite CW structure of dimension $n$.
Choose a basepoint in $B$.
By Lemma~\ref{lem:stable_Thom_equivalence}, there is
a pointed homeomorphism
$
\Sigma^{n-1}\operatorname{Th}(L_{u+v})
\cong \Sigma^{n-1}\operatorname{Th}(L_{u-v}).$
Applying Lemma~\ref{lem:projective_bundle_suspension}
and suspending a further $n-2$ times, we obtain
\[
\begin{aligned}
\Sigma^{n-1}P_{u+v}
&\simeq
\Sigma^{n-1}B\vee
\Sigma^{n-1}\operatorname{Th}(L_{u+v})\\
&\cong
\Sigma^{n-1}B\vee
\Sigma^{n-1}\operatorname{Th}(L_{u-v})\\
&\simeq
\Sigma^{n-1}P_{u-v},
\end{aligned}
\]
as required.
\end{proof}

\section{Toric manifolds and quasitoric manifolds}\label{sect:toric and quasitoric}

In this section, we  quickly review the  fundamental facts of toric and quasitoric manifolds on their cohomology and bundle projectivisation.
 These descriptions will be used to verify the conditions of Theorem~\ref{th:main'} for our examples in Section~\ref{section:examples}.

A toric manifold of complex dimension $n$  is a normal complete smooth complex algebraic
variety with an action of an $n$-dimensional algebraic torus $(\C^*)^n$ having
a dense orbit. There is a one-to-one correspondence between toric manifolds and smooth complete fans in $\R^n$, with respect to the lattice $\Z^n$. Thus we denote by $M_\Sigma$ a toric manifold with its defining fan $\Sigma$.

A quasitoric manifold of real dimension $2n$ is a closed smooth manifold equipped with a locally standard $(S^1)^n$-action whose orbit space is a simple polytope $Q$ of real dimension $n$. According to~\cite{DJ91}, a quasitoric manifold can be constructed by a combinatorial data $(Q,\lambda)$, where
$\lambda\colon \{F_1,\ldots,F_m\}\to\mathbb Z^n$
is a \emph{characteristic function} on the facets of $Q$ (that is, whenever $n$ facets meet at a vertex, their images under~$\lambda$ form a basis of $\mathbb Z^n$). We denote the resulting quasitoric manifold by~$M_{Q,\lambda}$.

Let $M$ be either a toric manifold $M_\Sigma$ of complex dimension $n$, or a quasitoric manifold~$M_{Q,\lambda}$ of real dimension $2n$.
Let $m$ denote the number of the rays of $\Sigma$ or facets of $Q$, respectively.
In the case of toric manifold, let $\rho_1,\ldots,\rho_m\in \Z^n$ be the primitive ray generators of $\Sigma$; in the case of quasitoric manifold, set $\rho_i=\lambda(F_i)$. Each ray or facet determines a characteristic submanifold of real codimension two. Denote its Poincar\'e dual by $D_i$, with orientations chosen compatibly with the vectors~$\rho_i$. 
\begin{theorem}[{\cite[\S10]{Dan78}, \cite[\S2]{Jur80}, \cite[\S4]{DJ91}}]\label{thm:coho of toric manifolds}
The integral cohomology of~$M$ is torsion-free and vanishes in odd degrees. Moreover,
there is an isomorphism of graded rings
\[
H^*(M)
\cong
\frac{\mathbb Z[D_1,\ldots,D_m]}{\mathcal I+\mathcal J},
\qquad \deg D_i=2.
\]
Here $\mathcal I$ is generated by the square-free monomials
$D_{i_1}\cdots D_{i_k}$ for which the  rays spanned by $\rho_{i_1}, \cdots \rho_{i_k}$ do not span a cone of $\Sigma$, or, in the quasitoric manifold case,
\[
F_{i_1}\cap\cdots\cap F_{i_k}=\varnothing.
\]
Writing $\rho_i=(\rho_{1i},\ldots,\rho_{ni})^T$, the ideal of linear relations is
\[
\mathcal J=
\left\langle
\sum_{i=1}^m\rho_{ji}D_i
\ \middle|\ 1\leq j\leq n
\right\rangle.
\]
\end{theorem}

The Betti numbers $b_i$ of $H^*(M)$ are determined by the combinatorial information of~$\Sigma$ or $Q$, respectively. 
In particular $b_2=m-n$, and the total $\sum_i b_i$ equals the number of maximal cones of $\Sigma$ in the toric case, or the number of vertices of $Q$ in the quasitoric case.

The degree two classes also determine the characteristic classes needed below. 
For a toric or quasitoric manifold, the total Pontryagin and Stiefel--Whitney classes are given by
\begin{equation}\label{eq Pontryagin and SW classes}
p(M)=\prod_{i=1}^m (1+D_i^2),\qquad w(M)=\prod_{i=1}^m (1+\overline{D_i}).
\end{equation}
For the toric case, these formulas follow from the Chern class formula in~\cite[Corollary 11.5]{Dan78}. For the quasitoric case,
see~\cite[Corollary 6.8(ii)]{DJ91} and~\cite[\S2.1.1]{CK11}.

Recall that each element $u\in H^2(M)$ determines a complex line bundle $L_u$ over~$M$. When $M$ is toric, the bundle $L_u$ can even be algebraic, due to \cite[\S10]{Dan78}.

\begin{proposition}
\label{prop:toric_projectivisations}
Let $L_{u_1},\ldots,L_{u_{k}}$ be  complex line bundles over $M$. 
\begin{enumerate}
   \item  \cite[Theorem~12.3.2, Proposition~7.3.3]{CLS11} If $M$ is a toric manifold,  so is $\mathbb P_M(L_{u_1}\oplus\cdots\oplus L_{u_k})$. If $M$ is projective, so is $\mathbb P_M(L_{u_1}\oplus\cdots\oplus L_{u_k})$.
\item \cite[Lemma~4.6]{AB23} If $M$ is a quasitoric manifold over a simple polytope $Q$, then $\mathbb P_M(L_{u_1}\oplus\cdots\oplus L_{u_k})$ is a quasitoric manifold over $Q\times\Delta^{k-1}$.   
\end{enumerate}
    
\end{proposition}

Under the hypotheses of Theorem~\ref{th:main'}, the manifold 
$N=\mathbb{CP}^{d_1}\times\cdots\times\mathbb{CP}^{d_k}$ is a product of complex projective spaces, so it is a projective toric manifold and a quasitoric manifold.
If~$B$ is a toric manifold, then
Proposition~\ref{prop:toric_projectivisations} implies that
$P_{u+v}\times N$ and $P_{u-v}\times N$ are toric manifolds
of complex dimension $4+\sum_id_i$. If $B$ is projective, so are
these manifolds. If $B$ is a quasitoric manifold over a $3$-dimensional simple polytope
$Q$, then the same proposition implies that these products
are quasitoric manifolds of real dimension $8+2\sum_id_i$ over the polytope
$Q\times I\times\Delta^{d_1}\times\cdots\times\Delta^{d_k}$.

\section{Examples}
\label{section:examples}

In this section, we construct a projective toric manifold of complex dimension~$3$ and a quasitoric manifold over a $3$-cube, together with classes $u,v,h \in H^2(B)$, satisfying those conditions in Theorem~\ref{th:main'}. 
We reuse the notation for the base and its classes separately in each example.

\subsection{A toric manifold  with \texorpdfstring{$b_2=4$}{}}
Consider the following fan $\Sigma$ with ray generators $\rho_1,\ldots,\rho_7\in \R^3$, which are column vectors of the matrix
\begin{align}\label{eq:1st fan}
\begin{pmatrix}
1&0&0&0&-1&0&0\\
0&1&0&0&-2&-1&1\\
0&0&1&-1&1&-2&-1
\end{pmatrix},
\end{align}
and maximal cones
\begin{align}\label{eq:1st cones}
    \begin{gathered}
\cone(\rho_1,\rho_2,\rho_3),\;\cone(\rho_1,\rho_2,\rho_7),\;\cone(\rho_1,\rho_3,\rho_6),\;\cone(\rho_1,\rho_4,\rho_6),\\
\cone(\rho_1,\rho_4,\rho_7),\;
\cone(\rho_2,\rho_3,\rho_5),\;\cone(\rho_2,\rho_5,\rho_7),\\
\cone(\rho_3,\rho_5,\rho_6),\;\cone(\rho_4,\rho_5,\rho_6),\;\cone(\rho_4,\rho_5,\rho_7).
    \end{gathered}
\end{align}
Each triple of ray generators corresponding to a maximal
cone has determinant $\pm1$, so $\Sigma$ is nonsingular.
\begin{lemma}
The fan $\Sigma$ is the normal fan
of the polytope
\[
Q=\{q\in\mathbb R^3:
\langle\rho_i,q\rangle\le\kappa_i,\ 1\le i\le7\},
\qquad
\kappa=(2,2,2,2,22,6,3).
\]
Therefore $\Sigma$ defines a projective
toric manifold $B$ of complex dimension~$3$.
\end{lemma}
\begin{proof}
Write $q=(q_1,q_2,q_3)\in \R^3$. The defining inequalities imply
\[
-28\le q_1\le2,\qquad -10\le q_2\le2,\qquad -2\le q_3\le2.
\]
So $Q$ is bounded.
Since each $\kappa_i$ is positive, the origin satisfies all defining inequalities strictly.
Hence $Q$ contains an open neighborhood of the origin. Thus $Q$ is $3$-dimensional.

Computing the intersections of three boundary planes
that satisfy all the defining inequalities gives
the following vertices and the indices of the facets
containing them
\[
\begin{array}{c|c@{\qquad}c|c}
\text{Vertex}&\text{Indices}&\text{Vertex}&\text{Indices}\\ \hline
(2,2,2)&123&(-24,2,2)&235\\
(2,2,-1)&127&(-27,2,-1)&257\\
(2,1,-2)&147&(-26,1,-2)&457\\
(2,-2,-2)&146&(-20,-2,-2)&456\\
(2,-10,2)&136&(0,-10,2)&356
\end{array}
\]
The normal cone at each vertex is generated by the outward
normals of the facets containing that vertex. These are precisely the maximal
cones in~\eqref{eq:1st cones}, so $\Sigma$ is the normal fan
of $Q$.
\end{proof}

The second Betti number of $H^*(B)$ is $4$. By Poincar\'e duality, we get 
\[
\left(H^0(B),\,H^2(B),\,H^4(B),\,H^6(B)\right)\cong \left(\Z,\,\Z^4,\,\Z^4,\,\Z\right).
\]
 By Theorem~\ref{thm:coho of toric manifolds}, the integral cohomology ring of $B$ is 
\[
  H^*(B)=\frac{\mathbb{Z}[D_i\mid 1\leq i\leq 7]}{\mathcal{I}+\mathcal{J}},
\]
where 
\begin{align*}
\begin{gathered}
    \mathcal{I}=\left\langle 
  D_1D_5,\; D_2D_4,\; D_2D_6,\; D_3D_4,\; D_3D_7,\; D_6D_7
\right\rangle,\\
\mathcal{J}=\left\langle 
  D_1-D_5,\; D_2-2D_5-D_6+D_7,\; D_3-D_4+D_5-2D_6-D_7
\right\rangle.
\end{gathered}
\end{align*}

Clearly, the cohomology classes  $D_4$, $D_5$, $D_5+D_6$, $D_7-D_5$ form a basis of $H^2(B)$, and we fix this basis.
Define
\begin{align}\label{eq:1st def of u and v}
    u= -D_4+D_5+D_7,\quad v=-D_4+4D_5-2D_7,\quad h=D_4.
\end{align}
The following two lemmas show that $u$, $v$, and $h$ satisfy Conditions (1)--(4) of Theorem~\ref{th:main'}.

\begin{lemma}\phantom{}
\begin{enumerate}
\item $(u+v)^2\ne0$, $(u+v)^2h=0$, and $h^3\ne0$.
\item $(u,v)\in\ZZ^{2,2}(B)\setminus\SZ^{2,2}(B)$.
\end{enumerate}
\end{lemma}
\begin{proof}
Equip $B$ with the orientation such that $\int_B D_1 D_2 D_3=1.$ Direct computation shows 
\begin{align*}
\begin{gathered}
\int_B (u+v)^2D_5=-9,\quad \int_B h^3=-12,\quad (u+v)^2 h=0,\\
uv=0,\qquad\overline{u+v} = {w_2}(B),\qquad 
\int_{B} \left(up_1(B)-u^3\right) =24.
\end{gathered}
\end{align*}
The first line proves part (1), and the second line proves part (2) by Proposition~\ref{prop:obstruction_SZ22}.
\end{proof}

\begin{lemma}\label{lem:form of automorphism of H*(B)}
Let $\varphi$ be a graded ring automorphism of $H^*(B)$.
With respect to the ordered degree $2$ basis 
$\{\text{$\alpha_1=D_4$, $\alpha_2=D_5$, $\alpha_3=D_5+D_6$, $\alpha_4=D_7-D_5$}\}$, 
write
\[
\varphi(\alpha_i)=\sum_{j=1}^4c_{ij}\alpha_j.
\]
Then the coefficient matrix $C=(c_{ij})$ is
$\varepsilon I_4$ or $\varepsilon S$, where
$\varepsilon\in\{\pm1\}$ and
\[
S=
\begin{pmatrix}
1&-2&2&1\\
0&0&1&0\\
0&1&0&0\\
0&0&0&-1
\end{pmatrix}.
\]
In particular, $\varphi(u-v)=\pm(u-v)$.
\end{lemma}

\begin{proof}
With respect to the degree $2$ basis $\{\alpha_1,\alpha_2,\alpha_3,\alpha_4\}$, 
we may identify a class $\sum_{i=1}^4a_i\alpha_i$ in~$H^2(B;\mathbb Q)$
with its column coordinate vector
$\alpha=(a_1,a_2,a_3,a_4)^T\in\mathbb Q^4$.
Then we have
\[
V_B(\alpha)=\int_B\alpha^3=-3F(\alpha),
\]
where
$F(\alpha)
=4a_1^3+3a_1^2a_2-2a_1^2a_3-2a_1^2a_4
-2a_1a_2a_3-2a_1a_2a_4+(a_2+a_3)a_4^2.$

Let~$\nabla F$ be the gradient of $F$. A direct calculation shows that the rational zero locus
of the gradient $\nabla F$ is
\[
\{\alpha \in\mathbb Q^4:\nabla F(\alpha)=0\}
=\mathbb Qe_2\cup\mathbb Qe_3,
\]
where $e_1,\ldots,e_4$ are the standard basis vectors
of $\mathbb Q^4$.

Since $\varphi$ acts on $H^6(B)\cong\mathbb Z$
by multiplication by some $\varepsilon\in\{\pm1\}$, we have
\[
F(\varphi(\alpha))=-\frac{1}{3}\int_B \varphi(\alpha^3)=-\frac{1}{3}\int_B \varepsilon\alpha^3=\varepsilon F(\alpha).
\]
The coordinate vector of $\varphi(\alpha)$ is
$C^T\alpha$. Setting $U=\varepsilon C^T$ and differentiating
$F(U\alpha)=F(\alpha)$, we obtain
\[
U^T\nabla F(U\alpha)=\nabla F(\alpha).
\]
Since $U$ is invertible, it preserves the rational zero
locus of $\nabla F$, and therefore either preserves each of the two lines
$\mathbb Qe_2$ and $\mathbb Qe_3$, or exchanges them.

\smallskip
\noindent
Case (I): $U$ preserves each line.
Since $U$ is unimodular, it has the form
\[
U=
\begin{pmatrix}
a&0&0&b\\
p&e&0&q\\
r&0&f&s\\
k&0&0&h
\end{pmatrix},
\qquad e,f\in\{\pm1\}.
\]
Regarding $F(U\alpha)=F(\alpha)$ 
as polynomials in $a_2,a_3$
with coefficients in $\mathbb Q[a_1,a_4]$, comparison of the $(a_2a_3)$-coefficients gives
\[
ef(a a_1+b a_4)=a_1,
\]
so $a=ef$ and $b=0$.
The coefficients of $a_2a_4^2$ and $a_3a_4^2$
then give $eh^2=fh^2=1$.
Thus $e=f=a=1$ and $h=\pm1$.

Comparing the coefficients of $a_1^2a_3$, $a_1a_3a_4$,
$a_1^2a_2$, $a_1a_2a_4$, and $a_4^3$ gives
\[
2p=k^2,\qquad 1-h+2kh=0.
\]
Since $h=\pm1$ and $p\in\mathbb Z$, these equations force
$h=1$ and $k=0$. The same coefficient comparisons give
\[
q=kh,\qquad 2r=k^2-2k,\qquad s=1-h+kh,
\]
so $p=q=r=s=0$.
Hence $U=I_4$.

\smallskip
\noindent
Case (II):   Similarly, we can analyse the case where $U$ exchanges the two lines.  
In this case,
\[
U=
\begin{pmatrix}
a&0&0&b\\
p&0&e&q\\
r&f&0&s\\
k&0&0&h
\end{pmatrix},
\qquad e,f\in\{\pm1\}.
\]
As in Case~(I), comparison of the $a_2a_3$-coefficients,
followed by the coefficients of~$a_2a_4^2$ and
$a_3a_4^2$, gives
$b=0$, $a=e=f=1$, and $h=\pm1$.  Comparing the coefficients of $a_1^2a_2$, $a_1a_2a_4$,
$a_1^2a_3$, $a_1a_3a_4$, and $a_4^3$ yields
\[
\begin{aligned}
2p=k^2-5,\qquad 1-h+2kh=0,\\
q=1+kh,\qquad 2r=k^2-2k+5,\qquad s=kh-h.
\end{aligned}
\]
Together with $h=\pm1$ and $p\in \Z$,
the first row implies $h=-1$ and $k=1$.
Substituting $h=-1$ and $k=1$, we obtain 
$p=-2,$ $r=2,$ $q=s=0.$
Therefore
\[
U=
\begin{pmatrix}
1&0&0&0\\
-2&0&1&0\\
2&1&0&0\\
1&0&0&-1
\end{pmatrix}.
\]

Since $C=\varepsilon U^T$, these two cases give
the two coefficient matrices in the statement.
In either case, $\varphi(\alpha_4)=\pm \alpha_4$.
As $u-v=3\alpha_4$, it follows that
$\varphi(u-v)=\pm(u-v)$.
\end{proof}


\subsection{A quasitoric manifold over \texorpdfstring{$I^3$}{}}
Let $I^3$ be a $3$-cube with opposite facet pairs
\[
(F_1, F_4),\quad (F_2, F_5), \quad (F_3,F_6).
\]
Define a characteristic function
$\lambda\colon\{F_1,\ldots,F_6\}\to\Z^3$
by taking $\lambda(F_i)$ to be the $i$th column of the matrix
\[
\begin{pmatrix}
 1&0&0&1&0&0\\
 0&1&0&-2&1&2\\
 0&0&1&1&1&1
 \end{pmatrix}.
\]
This matrix belongs to the family considered in~\cite[Definition~4.11]{Ha}. 
Let $B$ be the quasitoric manifold associated with $(I^3,\lambda)$.

The cohomology ring $H^*(B)$ is isomorphic to
\begin{equation}\label{eq cohomology ring M 1}
\frac{\Z[D_1,\ldots,D_6]}{\mathcal{I}+\mathcal{J}}, \quad \deg D_i=2,
\end{equation}
where 
\[\begin{split}
\mathcal{I}&=\langle D_1D_4, D_2D_5, D_3D_6\rangle,\\
\mathcal{J}&=\langle D_1+D_4, ~ D_2-2D_4+D_5+2D_6,~ D_3+D_4+D_5+D_6\rangle.
\end{split}\]

Then the formula~\eqref{eq cohomology ring M 1} becomes
\begin{equation}\label{eq cohomology ring M 2}
H^*(B)\cong
\frac{\Z[D_4,D_5,D_6]}
     {\langle D_4^2,\ D_5(2D_4-D_5-2D_6),\ D_6(D_4+D_5+D_6)\rangle}.
\end{equation}
The classes $D_4D_5$, $D_4D_6$ and $D_5D_6$ form a basis of $H^4(B)$, and $D_4D_5D_6$ generates $H^6(B)$. 

\begin{lemma}\label{lem:baseauto}
Let $\varphi\colon H^*(B)\to H^*(B)$ be an isomorphism of graded rings. 
Then its restriction $\varphi\colon H^2(B) \to H^2(B)$ is either identity $\mathrm{id}$ or $-\mathrm{id}$.  
\end{lemma}
\begin{proof}
Note that the set $\{\alpha_1=D_4,~ \alpha_2=D_5+D_6,~ \alpha_3=D_6\}$ is also a basis of~$H^2(B)$. 
For $1\leq i\leq 3$, it suffices to prove $\varphi(\alpha_i)=\varepsilon\alpha_i$ for some $\varepsilon\in \{\pm 1\}$. 

Substituting $D_4=\alpha_1$, \(D_5=\alpha_2-\alpha_3\) and \(D_6=\alpha_3\) into the relations of~\eqref{eq cohomology ring M 2} gives
\begin{equation}\label{eq relations}
\alpha_1^2=0,\quad
\alpha_2\alpha_3=-\alpha_1\alpha_3,\quad
\alpha_2^2-\alpha_3^2=2\alpha_1\alpha_2-2\alpha_1\alpha_3.
\end{equation}
Suppose $\varphi(\alpha_i)=\sum_{j=1}^3 c_{ij}\alpha_j$ for some integers $c_{ij}.$
Since $\alpha_1^2=0$, comparing the coefficients of $\varphi(\alpha_1)^2=0$ yields $c_{12}=c_{13}=0.$
Unimodularity implies $c_{11}=\pm 1.$

Since every square-zero element in $H^2(B)$ is a multiple of $\alpha_1$, the isomorphism $\varphi$ descends to a graded ring automorphism $\theta$ of
\[
A=H^*(B)/\langle \alpha_1\rangle\cong \Z[\alpha_2,\alpha_3]/\langle \alpha_2^2-\alpha_3^2, \alpha_2\alpha_3\rangle
\]
such that $\theta(\alpha_i)=c_{i2}\alpha_2+c_{i3}\alpha_3$ for $i=2,3.$ Then \cite[Lemma 3.3]{Ha} implies
\[
\mathrm{Aut}(A)=\left\{
\begin{pmatrix}
    \pm 1 & 0\\
    0 & \pm 1
\end{pmatrix}, \quad 
\begin{pmatrix}
0 & \pm 1 \\
\pm 1 & 0
\end{pmatrix}.
\right\}
\]
Consequently, there are two cases:
\[\begin{split}
\text{(I)} ~ \varphi(\alpha_2)&=c_{21}\alpha_1+c_{22} \alpha_2  \text{ and } \varphi(\alpha_3)= c_{31}\alpha_1+c_{33}\alpha_3,\quad c_{22},c_{33}=\pm 1; \\
\text{(II)} ~ \varphi(\alpha_2)&=c_{21}\alpha_1+c_{23}\alpha_3
\text{ and }
\varphi(\alpha_3)=c_{31}\alpha_1+c_{32}\alpha_2,\quad c_{23}, c_{32}=\pm 1.
\end{split}\]

In Case (I), substituting $\varphi(\alpha_2)$ and $\varphi(\alpha_3)$ to the relations of~\eqref{eq relations} gives
\[
c_{21}=c_{31}=0 \text{ and } c_{22}=c_{33}=c_{11}.
\]
In Case (II), the right two relations  of~\eqref{eq relations} give
\[
c_{21}=-c_{11}, \quad c_{31}=c_{32}, \quad
4=2c_{11} c_{32},
\]
which is impossible. Hence $\varphi(\alpha_i)=c_{11}\alpha_i$ for all $1\leq i\leq 3$, where $c_{11}\in \{\pm 1\}.$
\end{proof}

\begin{lemma}\label{lemma conditions check}
The classes $u=2D_4+2D_5+D_6$, $v=3D_4-D_5-3D_6$ and
$h=D_4-2D_5-D_6$ satisfy Conditions (1)--(4) of Theorem~\ref{th:main'}.
\end{lemma}
\begin{proof}
The characteristic class formulas for quasitoric manifolds give
\[
w_2(B)=\overline{D_4+D_5},
\qquad
p_1(B)=4D_4D_5-12D_4D_6-6D_5D_6.
\]
Choose the orientation of $B$ so that $\int_B D_4D_5D_6=1.$
It is a direct computation that
\[
\begin{gathered}
uv=0,\qquad (u+v)^2h=0, \qquad \int_B h^3=-15, \qquad
\int_B(u+v)^2\, D_4=-10,\\
 \overline{u+v}= w_2(B), \qquad
\int_B\bigl(up_1(B)-u^3\bigr)=24.
\end{gathered}
\]
The first row establishes  Conditions (2) and (3), and the second row establishes Condition~(1) by Proposition~\ref{prop:obstruction_SZ22}. Condition (4) follows from Lemma~\ref{lem:baseauto}.
\end{proof}

\begin{corollary}
For every $n\geq 4$, there exist quasitoric manifolds over the cube $I^n$ with isomorphic integral cohomology rings that are not homotopy equivalent.  
\end{corollary}
\begin{proof}
Let $B$ be the quasitoric manifold over $I^3$ constructed above. 
By Lemma~\ref{lemma conditions check} and Theorem~\ref{th:main'}, the manifolds
\[
 P_{u-v}\times(\mathbb{CP}^1)^{n-4},
\qquad
 P_{u+v}\times(\mathbb{CP}^1)^{n-4}
\]
have isomorphic integral cohomology rings but are not homotopy equivalent.
The projective bundles $P_{u-v}$ and $P_{u+v}$ are quasitoric manifolds over $I^3\times I=I^4$; see Proposition~\ref{prop:toric_projectivisations}\,(2). 
Consequently, both products are quasitoric manifolds over~$I^n$.
\end{proof}

\end{document}